\documentclass[reqno, 10pt]{amsart}
\usepackage{amssymb}
\usepackage{amsfonts}
\usepackage{txfonts}
\usepackage{mathrsfs}
\usepackage{mathrsfs}
\usepackage{amssymb,url}
\usepackage{graphicx}
\usepackage{cite}
\usepackage{enumerate}

\allowdisplaybreaks[4]
\theoremstyle{plain}
\newtheorem{theorem}{Theorem}

\newtheorem{lemma}{Lemma}

\newtheorem{proposition}{Proposition}

\newtheorem{definition}{Definition}

\usepackage[colorlinks,linkcolor=blue,urlcolor=red,linktocpage]{hyperref}

\begin{document}

\title{The Minkowski problem of $p$-affine dual curvature measures}

\author[Y. Lin]{Youjiang Lin}
\address[Y. Lin]{School of Mathematical Sciences, Hebei Normal University, Shijiazhuang, Hebei, 050024, China}
 \email{\href{mailto: YOUJIANG LIN
<yjlin@hebtu.edu.cn>}{yjlin@hebtu.edu.cn}}

\author[Y. Wu]{Yuchi Wu}
\address[Y. Wu]{School of Mathematical Sciences,  Shanghai Key Laboratory of PMMP, East China Normal University, Shanghai 200241, China}
 \email{\href{mailto: Yuchi Wu
<ycwu@math.ecnu.edu.cn>}{ycwu@math.ecnu.edu.cn}}


\begin{abstract}
For $p\in (-\infty,0)\cup(0,1)$  and a convex body $K\subset\mathbb{R}^n$ with the origin in its interior, we construct the family of $p$-affine dual curvature measures $\mathcal{I}_p(K,\cdot)$ with respect to $K$.
The affine-invariant measure $\widetilde{C}_{n-1}^{\mathrm{a}}(K, \cdot)$ given in the paper \cite{CLWX25} is the limit case of $\mathcal{I}_p(K,\cdot)$ as $p\rightarrow 1^-$. The classical cone-volume measure is the limit case of the affine measures $2|p|\mathcal{I}_p(K,\cdot)/(n^2V(\mathrm{I}_pK))$ when $p\rightarrow 0$ and $V(K)=2$, where $\mathrm{I}_pK$ denotes the $L_p$ intersection body of $K$. The Minkowski problems for the $p$-affine dual curvature measures are proposed and studied. Specifically, we give a sufficient condition for the existence of a solution to the even Minkowski problem for $p$-affine dual curvature measure. Moreover, a necessary condition is given when $p\in (0,1)$. The smooth case of this Minkowski problem is equivalent to solving a new type of partial differential equations with respect to $p$-cosine transforms.
\end{abstract}


\footnote{The First author is supported by National Natural Science Foundation of China NSFC 12371137 and NSFC 11971080; The second author is supported by  Science and Technology Commission of Shanghai Municipality 22DZ2229014 and Youth Fund of the National Natural Science Foundation of China NSFC 12401067.}

\maketitle	
	
	\section{Introduction}
	\label{intro}
Minkowski problem is one of the main problems in the Brunn-Minkowski theory.
The classical Minkowski problem states that: what are the necessary and sufficient conditions on a finite Borel measure $\mu$ on $S^{n-1}$ so that $\mu$ is the surface area measure $S(K,\cdot)$ of a convex body $K$ in $\mathbb{R}^n$? Minkowski \cite{Minkowski97,Minkowski03} posed and fully solved the discrete version of this problem. The general case was solved independently by Aleksandrov\cite{Aleksandrov38,Aleksandrov39}  and Fenchel-Jessen \cite{FJ38}.

The $L_p$ surface area measure introduced by Lutwak \cite{Lutwak93} is a fundamental notion in the $L_p$ Brunn-Minkowski theory, which contains the classical surface area measure (the case $p=1$) and the cone-volume measure (the case $p=0$)  as two important special cases. The $L_p$ Minkowski problem, posed by
Lutwak (see, e.g., \cite{Lutwak93}), asks for necessary and sufficient conditions that would guarantee
that a given measure on the unit sphere would be the $L_p$ surface area measure of a convex
body; see, e.g., \cite{Lutwak93,LO95,LYZ04,CW06,BLYZ13,Chen06,HLX15,Zhu1501}. The $L_p$ Minkowski problem is interesting for all real $p$ and contains the classical Minkowski problem (the case $p=1$) and the log-Minkowski problem (the case $p=0$) as two important special cases.  The solutions to the $L_p$ Minkowski problem have important application to affine isoperimetric inequalities; see, e.g.,\cite{Zhang99,LYZ00,LYZ02,LYZ021,CLYZ09,HS09,HS0901,HSX12}.

Corresponding to the Brunn-Minkowski theory, there is a parallel dual Brunn-Minkowski theory that was initiated by Lutwak in 1975 (see, e.g., \cite{Schneider14}). In the remarkable paper \cite{HLYZ16}, Huang-Lutwak-Yang-Zhang discovered a series of fundamental
dual curvature measures of convex bodies in $\mathbb{R}^n$. They studied the corresponding
dual Minkowski problems and proved the existence of solutions to the even dual
Minkowski problems. See
e.g., \cite{BLYZZ19,GHXY19,GHXY20,HLYZ16,HZ18,Zhao17,Zhao18} for recent works for dual Minkowski problems.

Recently, Lutwak-Xi-Yang-Zhang \cite{LXYZ} constructed a family of new geometric measures in integral geometry. They are the differentials of chord-power integrals, and are called simply chord measures. The related Minkowski problem is called the chord Minkowski problem. What's called the $(n-2)$-th Christoffel-Minkowski problem (see e.g. \cite{GM03}) is an important case of the chord Minkowski problem. More achievements on Minkowski's problems can be found in the literatures \cite{BLYZZ20,Livshyts10,HXZ21,LXYZ}.


%

The affine geometry of convex bodies studies geometric functionals that are invariant under affine and centro-affine transformations. During
the past half century, various new affine invariants were studied , see e.g. \cite{Lutwak88,Lutwak96,Meyer00,SW04,Lutwak90,Lei98,LYZ00,LYZ0001,LYZ02,WY10,Zhao16,Zhu15}. Some affine invariants have been characterized by additivity properties, see for example \cite{HP14,Ludwig03,Lutwak91,Ludwig06,Ludwig10}. Cone-volume measure is the only  $\operatorname{SL}(n)$
invariant measure among the $L_p$-surface area measures. Corresponding the log-Minkowski problem was solved by B\"or\"oczky-Lutwak-Yang-Zhang in \cite{BLYZ13} for symmetric convex bodies. For the case where $\mu$ is a discrete non-symmetric measure, the log-Minkowski problem was studied
by Stancu \cite{Stancu02,Stancu03}, Zhu \cite{Zhu14}. Chen-Li-Zhu \cite{CLZ18} established the existence of solutions to the log-Minkowski problem for non-symmetric measures.

Dual affine quermassintegrals were proposed by Lutwak, see \cite{Schneider14}. They were defined by letting $\widetilde{\Phi}_0(K)=V(K)$, and $\widetilde{\Phi}_n(K)=\omega_n$, while for $m=1, \ldots, n-1$,
$$
\widetilde{\Phi}_{n-m}(K):=\frac{\omega_n}{\omega_m}\left(\int_{G(n, m)}V_m(K \cap \xi)^n d \nu_m(\xi)\right)^{1 / n},
$$
where $\omega_k$ denotes the volume of $k$-dimensional unit ball and  $G(n, m)$ denotes the Grassmannian of $m$-dimensional subspaces of $\mathbb{R}^n$.
 Using the dual affine quermassintegrals, Cai-Leng-Wu-Xi \cite{CLWX25} presented the affine invariant measures $\widetilde{C}_{m}^{\mathrm{a}}(K, \cdot)$ (see Section \ref{S2.3} for precise definition) and proved the following variation formulas. For a convex body $K \subset \mathbb{R}^n$ that contains the origin in its interior, and for each $f \in C\left(S^{n-1}\right)$,
$$
\left.\frac{d}{d t}\right|_{t=0}\left[\widetilde{\Phi}_{n-m}\left([K, f]_t\right)^n\right]=n\left(\frac{\omega_n}{\omega_m}\right)^n \int_{S^{n-1}} f(v) d \widetilde{C}_m^{\mathrm{a}}(K, v),
$$
where for sufficiently small $|t|$, the logarithmic Wulff shapes $[K, f]_t$ is given in (\ref{W2.9}).
Cai-Leng-Wu-Xi \cite{CLWX25} attacked Minkowski problems for the  affine-invariant measures $\widetilde{C}_{m}^{\mathrm{a}}(K, \cdot)$. When $m=n$, the Minkowski problem with respect to the measure $\widetilde{C}_{n}^{\mathrm{a}}(K, \cdot)$ is the log-Minkowski problem.

 When $m=n-1$, the dual affine quermassintegral $\widetilde{\Phi}_1$ is closely related to the intersection body $\text {IK }$ (first defined and named in \cite{Lutwak88} by Lutwak) of the star body $K \subset \mathbb{R}^n$, which can be defined by
$$
\rho_{\text {IK }}(u)=V_{n-1}\left(K \cap u^{\perp}\right).
$$
 A simple computation shows that the volume of the intersection body of $K$ is related to $\widetilde{\Phi}_1(K)$ by:
$$
V(\mathrm{I} K)=\frac{\omega_{n-1}^n}{\omega_n^{n-1}} \widetilde{\Phi}_1(K)^n .
$$

For $p\in (-\infty,0)\cup(0,1)$, the {\it $L_p$ intersection body} of an origin-symmetric star body $K\subset \mathbb{R}^n$ is defined as  the star body $\mathrm{I}_p K$ with radial function
\begin{eqnarray}\label{DIP}
\rho\left(\mathrm{I}_p K, x\right)=\left(\frac{1-p}{2}\int_K|x \cdot y|^{-p} d y\right)^{1/p}, \quad x \in \mathbb{R}^{n}\setminus\{o\},
\end{eqnarray}
where $x \cdot y$ is the usual inner product of $x, y \in \mathbb{R}^n$, and integration is with respect to Lebesgue measure. Up to normalization, $L_p$ intersection body equals the polar $L_{-p}$ centroid body given by Lutwak-Zhang \cite{LZ97}.
%
%
%
 Note that the definition of $L_p$ intersection bodies is slightly different from the definition given in \cite{Haberl} by replacing $1/\Gamma(1-p)$ by $(1-p)/2$. When $p\rightarrow 1^-$, the $L_p$ intersection body $\mathrm{I}_pK$ converges to the classical intersection body $\mathrm{I}K$.

In this paper, by computing the variational formula of $L_p$ intersection bodies for $p\in (-\infty,0)\cup(0,1)$, we obtain a class of new affine-invariant measures $\mathcal{I}_p(K,\cdot)$ (see Definition \ref{DIPK}). Specially, for $K \in \mathcal{K}_o^n$, $f\in C\left(S^{n-1}\right)$ and let $K _ { t }$ be the logarithimic family of Wulff shapes generated by $h_K$ and $f$.
Then,
\begin{eqnarray}\label{varfor}
\left.\frac{d}{d t}\right|_{t=0}V\left(\mathrm{I}_pK_t\right)=
		\operatorname{sgn}(p)\int_{S^{n-1}}
f(u)d\mathcal{I}_p(K,u),
\end{eqnarray}
where $\operatorname{sgn}$ denotes the sign function.
The affine-invariant measure $\widetilde{C}_{n-1}^{\mathrm{a}}(K, \cdot)$ given in \cite{CLWX25} is the limit case of $\mathcal{I}_p(K,\cdot)$ when $p\rightarrow 1^-$. Moreover, we prove that the classical cone-volume measure $V_K(\cdot)$ is the limit case of the measure $2|p|\mathcal{I}_p(K,\cdot)/(n^2V(\mathrm{I}_pK))$ when $p\rightarrow 0$ and $V(K)=2$.


\begin{definition}\label{DIPK}
For any spherical Borel set $\eta\subset S^{n-1}$, ${\boldsymbol\alpha}_K^{\ast}(\eta)$ denotes its reverse radial Gauss image (see Section \ref{Sct2} for precise definition). Let $p\in (-\infty,0)\cup (0,1)$ and $K \in \mathcal{K}_o^n$. The {\it $p$-affine dual curvature measure}, $\mathcal{I}_p(K,\cdot)$, is defined by
\begin{eqnarray}\label{IPK}
\mathcal{I}_p(K,\eta)=\frac{1-p}{2|p|}\int_{{\boldsymbol\alpha}_K^{\ast}(\eta)}\rho^{n-p}_K(v)T_{-p}\rho^{n-p}_{\mathrm{I}_pK}(v)dv.
\end{eqnarray}
\end{definition}

Here for $p > - 1 ,$  $T _ { p } f$ denotes the {\it $p$-cosine transform}  of a function $f \in C \left( S ^ { n - 1 } \right)$, which
is defined by
\begin{eqnarray}\label{1.2}
T_{p}f(u)=\int_{S^{n-1}}f(v)|u\cdot v|^{p}dv,~~\text{for}~u\in S^{n-1}.
\end{eqnarray}

By (\ref{DIP}), (\ref{IPK}) and  (\ref{1.2}), we have
\begin{eqnarray}\label{2.3}
\mathcal{I}_p(K,\eta)=\frac{n-p}{|p|}\int_{{\boldsymbol\alpha}_K^{\ast}(\eta)}\rho^{n-p}_K(v)\rho^{p}_{\mathrm{I}_p^2K}(v)dv,
\end{eqnarray}
where $\mathrm{I}_p^2K$ denotes the $L_p$ intersection body of $\mathrm{I}_pK$.
Using (\ref{2.3}) and polar coordinate transformation, we conclude that
\begin{eqnarray}\label{np2}
\mathcal{I}_p(K,\eta)=\frac{(n-p)^2}{|p|} \int_{x \in K, x /|x| \in {\boldsymbol\alpha}_K^{\ast}(\eta)}\rho^p_{\mathrm{I}^2_pK}(x)dx.
\end{eqnarray}

\noindent{\bf Affine Minkowski problem for $\mathcal{I}_p(K,\cdot)$.} Suppose $p\in (-\infty,0)\cup (0,1)$ is fixed. Find necessary and sufficient
conditions of a finite Borel measure $\mu$ on $S^{n-1}$ so that $\mu=\mathcal{I}_p(K,\cdot) $ for some convex body $K\in \mathcal{K}_o^n.$

\medskip

When the measure $\mu$ has a density function $g: S^{n-1}\rightarrow\mathbb{R}$, the Minkowski problem for $\mathcal{I}_p(K,\cdot)$ is equivalent to the study of the following Monge-Amp\`{e}re type equation on $S^{n-1}$:
$$c(n,p)T_{-p}\circ T_{-p}^{(\frac{n-p}{p})}\circ\mathcal{D}^{(-n+p)}(h)\left(\frac {\nabla h+h_\iota }{|\nabla h +h_\iota|}\right)(v)\cdot|\nabla h(v)+h(v)v|^{-p}h(v)\operatorname{det}\left(\nabla^{2}h(v)+h(v)I\right)=g(v),$$
where $h:S^{n-1}\rightarrow(0,\infty)$ is the unknown function to be found, $\nabla h$ and $\nabla ^ { 2 } h$ denote the spherical
gradient and spherical Hessian, $h_\iota(w):=h(w)w,\forall~ w\in S^{n-1}$, $I$ is the identity map, $T_{-p}$ denotes the $(-p)$-cosine transform, and the duality operator $\mathcal{D}$ and the constant $c(n,p)$ are given by
$$\mathcal{D}h(u)=\max_{v\in S^{n-1}}\frac{u\cdot v}{h(v)},\;\;\;{\rm and}\;\;\;c(n,p)=\frac{n-p}{|p|}\left(\frac{1-p}{2(n-p)}\right)^{\frac{n}{p}}.$$

The following theorem gives a sufficient condition for the existence of a solution to the even Minkowski problem for $p$-affine dual curvature measure. We mainly use a variational
approach, where the functional is the sum of the volume of $L_p$ intersection body and an entropy integral. We first prove inequalities about the radial functions of $L_p$ intersection body and intersection body, see Lemma \ref{T2}. Then, we give an estimate of the entropy
integral which is given in \cite{BLYZZ19}, see Lemma \ref{L7.3}. Finally, we
give a  estimate of the volume of $L_p$ intersection bodies for a
carefully chosen barrier convex body.  A finite Borel measure $\mu$ on $S^{n-1}$ is said to satisfy the {\it strict subspace concentration inequality} if, for every subspace $\xi$ of $\mathbb{R}^n$ with $0<\operatorname{dim} \xi<n$,
\begin{eqnarray}\label{1.4}
\frac{\mu\left(\xi \cap S^{n-1}\right)}{ \mu\left(S^{n-1}\right) } < \frac{1}{n}\operatorname{dim} \xi.
\end{eqnarray}

\begin{theorem}
Let $p\in (-\infty,0)\cup (0,1)$ and $\mu$ be a non-zero even finite Borel measure on the unit sphere $S^{n-1}$. If $\mu$ satisfies the strict subspace concentration inequality (\ref{1.4}), then there exists an origin-symmetric convex body $K$ in $\mathbb{R}^n$ such that $\mu=\mathcal{I}_p(K,\cdot)$.
\end{theorem}

 In \cite{BLYZ15},  B\"or\"oczky-Lutwak-Yang-Zhang gave the necessary
and sufficient conditions for the existence of a solution to the even logarithmic
Minkowski problem. Specially, a non-zero finite even Borel measure on the unit sphere $S^{n-1}$ is the cone-volume measure of an origin-symmetric convex body in $\mathbb{R}^n$ if and only if it satisfies the subspace concentration condition. The studies of subspace concentration inequality of the cone-volume measure for origin-symmetric convex polytopes dates back to He-Leng-Li \cite{HLL06}, Henk-Sch\"urmann-Wills \cite{HSW05} and Xiong \cite{Xiong10}. Later, Henk-Linke \cite{HL14} proved that the cone-volume measure of a polytope with centroid at the origin satisfies the subspace concentration inequality. On the generalized $p$-th dual curvature measure, B\"or\"oczky-Henk-Pollehn \cite{BHP18} proved a tight subspace concentration inequality for the dual curvature measures of a symmetric convex body. Along the same ideas as that in \cite{BHP18}, we give a necessary condition for the existence of a solution to the $p$-affine dual curvature measure for $0<p<1$.
\begin{theorem}\label{T1.2}
Let $0<p<1$ and $K\subset\mathbb{R}^n$ be an origin-symmetric convex body. Then for every subspace $\xi$ of $\mathbb{R}^n$ with $0<\operatorname{dim} \xi<n$, the $p$-affine dual curvature measure $\mathcal{I}_p(K,\cdot)$ satisfies that
\begin{eqnarray}\label{1.7}
\frac{\mathcal{I}_p\left(K,\xi \cap S^{n-1}\right)}{\mathcal{I}_p\left(K,S^{n-1}\right)}<\frac{\operatorname{dim} \xi}{n-p}.
\end{eqnarray}
\end{theorem}

\section{Preliminaries} \label{Sct2}

We develop some notations and, for quick later references, list some basic facts about convex
bodies. Good general references for the theory of convex bodies are provided by the books of Gardner \cite{Gardner}, Gruber \cite{Gruber07} and Schneider \cite{Schneider14}.

 Let $\mathbb{R}^n$ denote $n$-dimensional Euclidean space with canonical inner product $x \cdot y$, for $x, y \in$ $\mathbb{R}^n$;
throughout we assume that $n
 \geq 2.$ Let $o$ denote the origin of $\mathbb{R}^n.$
 Write $|x|=\sqrt{x \cdot x}$ for the norm of $x$. Let $S^{n-1}$ denote the unit sphere in $\mathbb{R}^n$. Let $B^n$ denote the unit ball centered at the origin of $\mathbb{R}^n$. The $n$-dimensional volume of the unit ball $B^n$ is denoted by $\omega_n$. Let $C(S^{n-1})$ denote the set of continuous functions defined on $S^{n-1}$. Let $C_e(S^{n-1})$ denote the set of even and continuous functions defined on $S^{n-1}$. Let $C^+(S^{n-1})$ denote the set of positive and   continuous functions defined on $S^{n-1}$.  Let $C_e^+(S^{n-1})$ denote the set of even, positive and continuous functions defined on $S^{n-1}$. The group of special linear transformations of $\mathbb{R}^n$ is denoted by $\operatorname{SL}(n)$. The $m$-dimensional Hausdroff measure is denoted by $\mathcal{H}^m$ and the $m$-dimensional Lebesgue measure is denoted by $\mathcal{L}^m$ . When considering an $m$-dimensional convex body $L$ that lies in an $m$-dimensional affine subspace of $\mathbb{R}^n$, we will write $V_m(L)$ rather than $\mathcal{L}^m(L)$. For $K\subset\mathbb{R}^n$, let $V(K)$ denote $V_n(K)$ for abbreviation. Let $K|\xi$ denote the orthogonal projection of $K\subset\mathbb{R}^n$ onto the affine subspace $\xi$ of $\mathbb{R}^n$. Let $\xi^{\perp}$ denote the orthogonal complementary space of the linear subspace $\xi$ of $\mathbb{R}^n$. Let $u^{\perp}$ denote the orthogonal complementary space of the unit vector $u\in S^{n-1}$. For $K\subset\mathbb{R}^n$, let $\operatorname{conv}K$ denote the convex hull of $K$. For a finite measure $\mu$ on $S^{n-1}$, we shall write $|\mu|$ for its total mass, that is $|\mu|=\mu\left(S^{n-1}\right)$. The sign function $\operatorname{sgn}:\mathbb{R}\to\mathbb{R}$ is defined as $\operatorname{sgn}(x)=x/|x|$ if $x\neq 0$ and $\operatorname{sgn}(0)=0.$

\subsection{Convex bodies and star bodies.}

The set of {\it convex bodies} (compact convex subsets with nonempty interiors) of $\mathbb{R}^n$ is denoted by $\mathcal{K}^n$, the set of convex bodies containing the origin in their interiors is denoted by $\mathcal{K}_o^n$, and $\mathcal{K}_e^n \subset \mathcal{K}_o^n$ denotes the set of origin-symmetric convex bodies.

For two sets $K,L\subset\mathbb{R}^n$, the {\em Minkowski sum} $K+L$ is defined by
$$K+L=\{x+y: x\in K, y\in L\}.$$
The scalar multiplication $\lambda K$ of $K$ is defined as
$$\lambda K=\{\lambda x:x\in K\},\;\;\lambda>0.$$
The polar body $K^*$ of $K\subset \mathcal{K}_o^n$ is defined as
$$K^*=\{y\in \mathbb{R}^n:x\cdot y\leq 1,\forall x\in K\}.$$

The {\it support function}, $h_K: S^{n-1} \rightarrow \mathbb{R}$, of a nonempty compact convex set $K \subset \mathbb{R}^n$, is the continuous function on the unit sphere $S^{n-1}$, defined by
\begin{equation}\label{hKv}
h_K(v)=\max \{x \cdot v: x \in K\} .
\end{equation}
A convex body is uniquely determined by its support function. The collection of convex bodies can be viewed as a metric space with the {\it Hausdroff metric}, where the Hausdorff distance between $K, L$ is defined by
\begin{eqnarray}
d_H(K,L):=\min\{t\geq 0:\;K\subset L+tB^n,\;L\subset K+tB^n\}.
\end{eqnarray}
For all $v \in S^{n-1}$, and all $\varphi \in \operatorname{SL}(n)$,
\begin{eqnarray}
h_{\varphi K}(v)=\left|\varphi^{\mathrm{t}} v\right| h_K\left(\left\langle\varphi^{\mathrm{t}} v\right\rangle\right)
\end{eqnarray}
where $\langle x\rangle=x /|x|$ for $x \in \mathbb{R}^n \backslash\{o\}$, and where $\varphi^t$ is the transpose of $\varphi$.

The {\it radial function}, $\rho_K: \mathbb{R}^n \backslash\{o\} \rightarrow \mathbb{R}$, of a compact star-shaped (with respect to the origin) set $K \subset \mathbb{R}^n$, is defined by
\begin{equation}\label{rhoK}
\rho_K(x)=\max \{t \geq 0: t x \in K\} .
\end{equation}
It is clear that
\begin{eqnarray}\label{rhoK}
\rho_{\varphi K}(x)=\rho_K(\varphi^{-1} x),\;\;x\in \mathbb{R}^n\backslash\{o\}, \;\; \varphi \in\operatorname{SL}(n).
\end{eqnarray}
The star-shaped set $K$ is said to be a star body, if $\rho_K$ is continuous and positive on the unit sphere $S^{n-1}$. The set of star bodies is denoted by $\mathcal{S}_o^n$, and $\mathcal{S}_e^n \subset \mathcal{S}_o^n$ denotes the set of origin-symmetric star bodies. For $K\in\mathcal{S}_o^n$, we have the following volume formula,
\begin{eqnarray}\label{volumeformula}
V(K)=\frac{1}{n}\int_{S^{n-1}}\rho^n_K(u)du,
\end{eqnarray}
where $du$ denotes the spherical Lebesgue measure.

The {\it radial distance} between $K,L\in\mathcal{S}_o^n$ is defined as
\begin{eqnarray}\label{e10}
d_R(K,L):=|\rho_K-\rho_L|_{\infty}=\max_{u\in S^{n-1}}|\rho_{K}(u)-\rho_L(u)|.
\end{eqnarray}

Let $h\in C^+\left(S^{n-1}\right)$, $f\in C\left(S^{n-1}\right)$ and $\delta>0$. Let $h_t\in C\left(S^{n-1}\right)$ be functions defined for each $t \in(-\delta, \delta)$ and each $v \in S^{n-1}$ by
\begin{eqnarray}
\log h_t(v)=\log h(v)+t f(v).
\end{eqnarray}\label{W2.8}
Denote by $\left[h_t\right]$ the Wulff shape determined by $h_t$,
\begin{eqnarray}\label{W2.9}
\left[h_t\right]=\left\{x \in \mathbb{R}^n: x \cdot v \leqslant h_t(v) \text { for all } v \in S^{n-1}\right\} .
\end{eqnarray}
We shall call $\left[h_t\right]$ a {\it logarithmic family of Wulff shapes formed by $(h, f)$}. On occasion, we shall write $\left[h_t\right]$ as $[h, f]_t$, and if $h$ happens to be the support function of a convex body $K\in\mathcal{K}_o^n$ perhans as $[K, f]_t$, or as $K_t$ for abbreviation.
\subsection{Radial Gauss image and its reverse.}\label{Sec2}
In \cite{HLYZ16}, Huang-Lutwak-Yang-Zhang gave the definitions of radial Gauss image and reverse radial Gauss image. For $K \in \mathcal{K}_o^n$ and $\eta \subset S^{n-1}$, the radial Gauss image $\boldsymbol{\alpha}_K(\eta)$ is the set of outer unit normals of $K$ at the boundary points $\rho_K(u) u$, for some $u \in \eta$; that is,
$$
\boldsymbol{\alpha}_K(\eta)=\bigcup_{u \in \eta}\left\{v\in S^{n-1}:\; \rho_K(u)u\cdot v=h_K(v)\right\} .
$$

The reverse radial Gauss image $\boldsymbol{\alpha}_K^*(\eta)$ is the set of all radial directions $u \in S^{n-1}$, such that an element in $\eta$ is an outer unit normal to $\partial K$ at the point $\rho_K(u) u \in \partial K$; that is,
$$
\boldsymbol{\alpha}_K^*(\eta)=\bigcup_{v \in \eta}\left\{u \in S^{n-1}:\; \rho_K(u) u \cdot v=h_K(v)\right\} .
$$

Denote $\omega_K \subset S^{n-1}$ as the set of $u \in S^{n-1}$ such that $\boldsymbol{\alpha}_K(\{u\})$ contains more than one point; that is, the point $\rho_K(u) u \in \partial K$ has more than one outer unit normal. We now define the radial Gauss map
$$
\alpha_K: S^{n-1} \backslash \omega_K \longrightarrow S^{n-1},
$$
satisfying $\boldsymbol{\alpha}_K(\{u\})=\left\{\alpha_K(u)\right\}$. It is well known that $\mathcal{H}^{n-1}\left(\omega_K\right)=0$ (see \cite{Schneider14}).

Denote $\eta_K \subset S^{n-1}$ as the set of all $v \in S^{n-1}$ for which $\boldsymbol{\alpha}_K^*(\{v\})$ contains more than one element; that is, there is more than one radial direction $u \in S^{n-1}$ such that $v$ is the unit outer normal of $\rho_K(u) u$. Similarly the reverse radial Gauss map is defined by
$$
\alpha_K^*: S^{n-1} \backslash \eta_K \longrightarrow S^{n-1},
$$
by $\boldsymbol{\alpha}_K^*(\{u\})=\left\{\alpha_K^*(u)\right\}$. The set $\eta_K$ is also of $\mathcal{H}^{n-1}$-measure zero (see \cite{Schneider14}).

\subsection{The affine dual curvature measure}\label{S2.3}
Given a spherical Borel measure $\mu ,$ and $\varphi \in \operatorname{SL}(n)$,
then $\varphi \mu ,$ its affine image, is defined (see \cite{BLYZ15}) by
\begin{equation}\label{defaffineimage}
\varphi\mu(\eta)=\mu\left(\left\langle\varphi^{-1}\eta\right\rangle\right),\quad \text{for~ each~ Borel}~ \eta \subset S^{n-1},
\end{equation}
where $\langle x \rangle=x/|x|,$ for $x\in\mathbb{R}^{n}\backslash\{o\} .$

 Suppose $m\in\{1,\ldots,n-1\}$, and suppose $K \in \mathcal{K}_o^n$.  Cai-Leng-Wu-Xi \cite{CLWX25} gave the definition of the {\it affine dual curvature measures} of $K$ by
$$
\widetilde{C}_m^{\mathrm{a}}(K, \eta)=\int_{{\boldsymbol \alpha}_K^*(\eta)} \rho_K(u)^m \int_{u \in \xi \in G(n, m)} V_m(K \cap \xi)^{n-1} d \xi d u,
$$
for each Borel set $\eta \subset S^{n-1}$, where $G(n, m)$ denotes the Grassmannian of $m$-dimensional subspaces of $\mathbb{R}^n$.  For each $\varphi \in \operatorname{SL}(n)$, the affine dual curvature measures of bodies satisfy the  affine contra-variant property:
$$
\widetilde{C}_m^{\mathrm{a}}(\varphi K, \cdot)=\varphi^{-\mathrm{t}} \widetilde{C}_m^{\mathrm{a}}(K, \cdot) .
$$

The {\it intersection body} of a star body $K\in\mathcal{S}_o^n$ is the centered star body $\mathrm{I} K$ defined by
\begin{eqnarray}\label{1.8}
\rho_{\mathrm{I} K}(u)=V_{n-1}\left(K \cap u^{\perp}\right)=\frac{1}{n-1} R\left(\rho_K^{n-1}\right)(u),
\end{eqnarray}
for each $u \in S^{n-1}$. Here $R$ denotes the spherical Radon transform,
for $f \in C\left(S^{n-1}\right)$, defined by
$$
R f(u)=\int_{S^{n-1} \cap u^{\perp}} f(v) d v,\;\;\;u \in S^{n-1}.
$$

The affine measure $\widetilde{C}_m^{\mathrm{a}}(K, \cdot)$ has a geometric interpretation by way of Zhang intersection bodies (see \cite{Zhang96}). In fact, for $m=n-1$, recalling (\ref{1.8}), it can be shown that, for $u \in S^{n-1}$,
$$
\int_{u\in\xi\in G(n,n-1)} V_{n-1}(K \cap \xi)^{n-1} d \xi=\frac{n-1}{n \omega_n} \rho_{\mathrm{I}^2 K}(u),
$$
where $\mathrm{I}^2 K=\mathrm{I}(\mathrm{I} K)$ is the intersection body of $\mathrm{I} K$. Therefore
\begin{eqnarray}\label{1.9}
\widetilde{C}_{n-1}^{\mathrm{a}}(K,\eta)=\frac{n-1}{n \omega_n} \int_{{\boldsymbol \alpha}_K^{\ast}(\eta)} \rho_K(u)^{n-1} \rho_{\mathrm{I}^2 K}(u) d u .
\end{eqnarray}

\section{Variation formula for $p$-affine dual curvature measure}

\begin{lemma}\cite{HLYZ16}\label{Huangcata}
Let $h\in C^+(S ^ { n - 1 })$ and $f \in C(S ^ { n - 1 })$.
Let $K _ { t }$ be the logarithimic family of Wulff shapes formed by $(h,f)$,
and abbreviate $K _ { 0 }$ by $K $. Then for almost all $u \in S ^ { n - 1 } $,
\begin{equation}\label{dKt}
  \lim _ { t \rightarrow 0 } \frac { \rho _ { K _ { t } } ( u )  - \rho _ { K } ( u )  } { t } =  \rho _ { K } ( u )  f \left( \alpha _ { K } ( u ) \right) .
\end{equation}
Moreover, there exist $\delta _ { 0 } > 0$ and $M _ { 0 } > 0$ so that
\begin{equation}\label{dKtbound}
\left| \rho _ { K _ { t } } ( u )  - \rho _ { K } ( u )  \right| \leq M _ { 0 } | t | ,\end{equation}
for all $u \in S ^ { n - 1 }$ and all $t \in \left( - \delta _ { 0 } , \delta _ { 0 } \right) $.

\end{lemma}

\begin{theorem}\label{dtrhoKt}
Let $p\in(-\infty,0)\cup(0,1)$, $K \in \mathcal{K}_o^n$, $f\in C\left(S^{n-1}\right)$ and let $K _ { t }$ be the logarithimic family of Wulff shapes generated by $h_K$ and $f$.
Then,
\begin{eqnarray}\label{varfor}
\lim_{t\rightarrow 0}\frac{V\left(\mathrm{I}_pK_t\right)-V\left(\mathrm{I}_pK\right)}{t}
			=\operatorname{sgn}(p)\int_{S^{n-1}}
f(u)d\mathcal{I}_p(K,u).
\end{eqnarray}
\end{theorem}

\begin{proof}
First, by (\ref{volumeformula}), (\ref{DIP}) and the polar coordinate transformation,
\begin{eqnarray}\label{IpK}
V(\mathrm{I}_pK)&=&\frac{1}{n}\int_{S^{n-1}}\rho_{\mathrm{I}_pK}(u)^ndu\nonumber\\
&=&\frac{1}{n}\int_{S^{n-1}}\left(\frac{1-p}{2}\int_K\left|x\cdot u\right|^{-p}dx\right)^{\frac{n}{p}}du\nonumber\\
&=&\frac{1}{n}\int_{S^{n-1}}\left(\frac{1-p}{2}\int_{S^{n-1}}\int_{0}^{\rho_K(v)}\left|v\cdot u\right|^{-p}r^{n-p-1}drdv\right)^{\frac{n}{p}}du\nonumber\\
&=&\frac{1}{n}\int_{S^{n-1}}\left(\frac{1-p}{2(n-p)}\int_{S^{n-1}}\left|v\cdot u\right|^{-p}\rho_K(v)^{n-p}dv\right)^{\frac{n}{p}}du.
\end{eqnarray}
Thus,  by (\ref{IpK}), (\ref{dKtbound}), Lebesgue dominated convergence theorem, (\ref{dKt}), Fubini's theorem, (\ref{1.2}) and (\ref{IPK}),
\begin{eqnarray}
&&\lim_{t\rightarrow 0}\frac{V\left(\mathrm{I}_pK_t\right)-V\left(\mathrm{I}_pK\right)}{t}\nonumber\\
&=&\frac{1-p}{2p}\int_{S^{n-1}}\int_{S^{n-1}}\rho^{n-p}_{\mathrm{I}_pK}(u)|v\cdot u|^{-p}\rho^{n-p}_{K}(v)f(\alpha_{K}(v))dvdu\nonumber\\
&=&\frac{1-p}{2p}\int_{S^{n-1}}\rho^{n-p}_{K}(v)f(\alpha_{K}(v))\int_{S^{n-1}}\rho^{n-p}_{\mathrm{I}_pK}(u)|v\cdot u|^{-p}dudv\nonumber\\
&=&\frac{1-p}{2p}\int_{S^{n-1}}\rho^{n-p}_{K}(v)f(\alpha_{K}(v))T_{-p}\rho_{\mathrm{I}_pK}^{n-p}(v)dv\nonumber\\
&=&\operatorname{sgn}(p)\int_{S^{n-1}}
f(u)d\mathcal{I}_p(K,u).\nonumber
\end{eqnarray}\qed
\end{proof}

\section{Affine contra-invariance of $\mathcal{I}_p(K, \cdot)$}

Suppose $\{ \mu ( K , \cdot ) \} _ { K \in \mathcal { K } ^ { n } }$ is a family of spherical Borel measures that are parametrized by
the set of convex bodies $K \in \mathcal { K } ^ { n } .$ We say $\mu$ is an {\it affine covariant family}, if for each $K \in \mathcal { K } ^ { n }$
and each $\varphi \in \operatorname{SL}(n) ,$
$$\mu ( \varphi K , \cdot ) = \varphi \mu ( K , \cdot ) ;$$
and say $\mu$ is {\it affine contra-variant family}, if for each $K \in \mathcal { K } ^ { n }$ and each $\varphi \in \operatorname{SL}(n) ,$
$$\mu ( \varphi K , \cdot ) = \varphi ^ { - \mathrm { t } } \mu ( K , \cdot ) ,$$
where $\varphi ^ { - \mathrm { t } }$ is the inverse of the transpose of $\varphi .$


\begin{lemma}\label{LL3.1}
Let $K\in\mathcal{K}_o^n$ and $\varphi\in \operatorname{SL}(n)$. Then
\begin{eqnarray}\label{IPvar}
\mathrm{I}_p(\varphi K)=\varphi^{-t} \mathrm{I}_pK.
\end{eqnarray}
\end{lemma}
\begin{proof}
By the definition of $L_p$ intersection body (\ref{DIP}), let $u\in S^{n-1}$ and $v=\varphi^tu/|\varphi^tu|$,
\begin{eqnarray}\label{rho1}
\rho\left(\mathrm{I}_p(\varphi K),u\right)^p=\frac{1-p}{2}\int_{\varphi K}\left|x\cdot u\right|^{-p}dx=\frac{1-p}{2}\left|\varphi^tu\right|^{-p}\int_K\left|y\cdot v\right|^{-p}dy.
\end{eqnarray}
Moreover, by (\ref{rhoK}) and (\ref{DIP}),
\begin{eqnarray}\label{rho2}
\rho\left(\varphi^{-t}\mathrm{I}_pK,u\right)^p=\rho\left(\mathrm{I}_pK,\varphi^{t}u\right)^p=\frac{1-p}{2}\int_{ K}\left|x\cdot\varphi^{t}u\right|^{-p}dx=\frac{1-p}{2}\left|\varphi^tu\right|^{-p}\int_K\left|x\cdot v\right|^{-p}dx.
\end{eqnarray}
By (\ref{rho1}), (\ref{rho2}) and the arbitrariness of $u\in S^{n-1}$, (\ref{IPvar}) is established. \qed
\end{proof}
\begin{proposition} The $p$-affine dual curvature measure $\mathcal{I}_p(K, \cdot)$ is affine contra-variant family, i.e., for any $\varphi \in \operatorname{SL}(n)$, we have
$$
\mathcal{I}_p(\varphi K, \eta)=\varphi^{-t}\mathcal{I}_p\left(K,\eta\right).
$$
\end{proposition}

\begin{proof} For $K \in \mathcal { K } _ { o } ^ { n }$ and $f \in C \left( S ^ { n - 1 } \right) ,$ let $[ K , f ] _ { t }$ be the logarithimic family of Wulff shapes.
By $\varphi \in \operatorname{SL}(n)$, Lemma \ref{LL3.1} and Theorem \ref{dtrhoKt}, we have
 \begin{eqnarray}\label{dtIpphiK}
&&\left. \frac { d } { d t } \right| _ { t = 0 } V\left(\mathrm{I}_p \left(\varphi [ K , f ] _ { t }\right) \right) = \left. \frac { d } { d t } \right| _ { t = 0 } V\left(\mathrm{I}_p [ K , f ] _ { t } \right)=\operatorname{sgn}(p)\int_{S^{n-1}}
f(u)d\mathcal{I}_p(K,u).
 \end{eqnarray}
Extend $f$ to be a 0-homogeneous function on $\mathbb { R } ^ { n } \backslash \{ o \} .$ Using (\ref{W2.8}) and (\ref{W2.9}), we can rewrite the set
 $\varphi [ K , f ] _ { t }$ as follows,
\begin{align*}
    \varphi [ K , f ] _ { t } &= \left\{ \varphi x : \log ( x \cdot v ) _ { + } \leq \log h _ { K } ( v ) + t f ( v ) , \quad \forall v \in S ^ { n - 1 } \right\} \\
   & = \left\{ y : \log ( y \cdot w ) _ { + } \leq \log h _ { \varphi K } ( w ) + t f \left( \varphi ^ { \mathrm { t } } w \right) , \quad \forall w \in S ^ { n - 1 } \right\} \\
   & = \left[ \varphi K , f \circ \varphi ^ { \mathrm { t } } \right] _ { t } .
\end{align*}
It follows from Theorem \ref{dtrhoKt} that
\begin{eqnarray}
\left. \frac { d } { d t } \right| _ { t = 0 } V\left(\mathrm{I}_p\left( \varphi [ K , f ] _ { t }\right)\right) = \left. \frac { d } { d t } \right| _ { t = 0 } V\left(\mathrm{I}_p \left[ \varphi K , f \circ \varphi ^ { \mathrm { t } } \right] _ { t } \right)=\operatorname{sgn}(p)\int_{S^{n-1}}
f\left( \left\langle \varphi ^ { \mathrm { t } } v \right\rangle \right) d\mathcal{I}_p(\varphi K,u).
\end{eqnarray}
Since $f$ is an arbitrary continuous function on $S ^ { n - 1 } ,$ and we are dealing with finite Borel
measures on $S ^ { n - 1 } ,$ the integrals of $f$ can be replaced by integrals of the characteristic function
of Borel sets. Thus, the above equation together with (\ref{dtIpphiK}) imply that
$$\int _ { S ^ { n - 1 } } 1 _ { \eta } ( v ) d \mathcal { I } _ { p }  ( K , v ) = \int _ { S ^ { n - 1 } } 1 _ { \eta } \left( \left\langle \varphi ^ { \mathrm { t } } v \right\rangle \right) d \mathcal { I } _ { p } ( \varphi K , v ) ,$$
for each Borel subset $\eta \subset S ^ { n - 1 } .$ Since for $v \in S ^ { n - 1 } ,$ we have $\left\langle \varphi ^ { \mathrm { t } } v \right\rangle \in \eta$ if and only if $v \in \left\langle \varphi ^ { - \mathrm { t } } \eta \right\rangle ,$
it follows that
$$\mathcal{I}_p ( K , \eta ) = \mathcal{I}_p \left( \varphi K , \left\langle \varphi ^ { - \mathrm { t } } \eta \right\rangle \right) = \varphi ^ { \mathrm { t } } \mathcal{I}_p ( \varphi K , \eta ) ,$$
where the right-hand equality follows from definition (\ref{defaffineimage}). This immediately yields the desired
fact that $\mathcal{I}_p ( \varphi K , \cdot ) = \varphi ^ { - \mathrm { t } } \mathcal{I}_p ( K , \cdot ) .$\qed
\end{proof}

\section{Some special cases of the $p$-affine dual curvature measures.}

\begin{proposition} Let $K \in \mathcal{K}_o^n$. Then the measures
$\mathcal{I}_p(K,\cdot)$ weakly converges to the measure $n\omega_n\widetilde{C}_{n-1}^{\mathrm{a}}(K,\cdot)$ given in (\ref{1.9}) as $p\rightarrow 1^{-}$.
\end{proposition}

\begin{proof}
 It is known that
\begin{eqnarray}\label{1.3}
\lim _ { p \rightarrow  1 ^- } \frac {  1-p } { 2 } T _ { -p } f = R f ,
\end{eqnarray}
for each $f \in C \left( S ^ { n - 1 } \right)$, see \cite{GZ} or \cite{Koldobsky97}. By \cite[Proposition 3.1]{GG99}, $\mathrm{I}_pK$ converges to $\mathrm{I}K$ in the radial distance when $p\rightarrow 1^-$, i.e.,
\begin{eqnarray}\label{IpKIK}
\lim_{p\rightarrow 1^-}d_R(\mathrm{I}_pK,\mathrm{I}K)=0.
\end{eqnarray}
For any spherical Borel set $\eta\subset S^{n-1}$, by (\ref{IPK}), (\ref{1.3}), (\ref{IpKIK}) and (\ref{1.9}),
\begin{eqnarray}
\lim_{p\rightarrow 1^{-}}\mathcal{I}_p(K,\eta)
&=&\lim_{p\rightarrow 1^{-}}\frac{1-p}{2p}\int_{{\boldsymbol\alpha}_K^{\ast}(\eta)}\rho^{n-p}_K(v)T_{-p}\rho^{n-p}_{\mathrm{I}_pK}(v)dv\nonumber\\
&=&\lim_{p\rightarrow 1^{-}}\frac{1}{p}\int_{{\boldsymbol\alpha}_K^{\ast}(\eta)}\rho^{n-p}_K(v)\frac{1-p}{2}T_{-p}\rho^{n-p}_{\mathrm{I}_pK}(v)dv\nonumber\\
&=&\int_{{\boldsymbol\alpha}_K^{\ast}(\eta)}\rho^{n-1}_K(v)R\rho^{n-1}_{\mathrm{I}K}(v)dv\nonumber\\
&=&(n-1)\int_{{\boldsymbol\alpha}_K^{\ast}(\eta)}\rho^{n-1}_K(v)\rho_{\mathrm{I}^2K}(v)dv\nonumber\\
&=&n\omega_n\widetilde{C}_{n-1}^{\mathrm{a}}(K,\eta).\nonumber
\end{eqnarray}
This is the desired conclusion. \qed
\end{proof}

\begin{proposition}Let $K \in \mathcal{K}_o^n$ with $\mathcal{C}^1$ boundary. Then
$\mathcal{I}_p(K,\cdot)$ converges to zero measure as $p\rightarrow -\infty$.
\end{proposition}

\begin{proof}
By (\ref{DIP}) and (\ref{2.3}),
\begin{eqnarray}\label{3.2}
\left(\frac{\rho_{\mathrm{I}_p^2K(v)}}{\rho_K(v)}\right)^p&=&\frac{(1-p)\int_{\mathrm{I}_pK}|v\cdot x|^{-p}dx}{2\rho^p_K(v)}\nonumber\\
&=&\frac{(1-p)\int_{S^{n-1}}\int_{0}^{\rho_{\mathrm{I}_pK}(u)}|v\cdot u|^{-p}r^{n-p-1}drdu}{2\rho^p_K(v)}\nonumber\\
&=&\frac{1-p}{2(n-p)}\int_{S^{n-1}}|v\cdot u|^{-p}\rho^{-p}_K(v)\rho^{n-p}_{\mathrm{I}_pK}(u)du\nonumber\\
&=&\frac{1-p}{2(n-p)}\int_{S^{n-1}}|v\cdot u|^{-p}\rho^{-p}_K(v)\left(\frac{1-p}{2(n-p)}\int_{S^{n-1}}|u\cdot w|^{-p}\rho_K(w)^{n-p}dw\right)^{-1}\rho^{n}_{\mathrm{I}_pK}(u)du\nonumber\\
&=&\int_{S^{n-1}}\left(\int_{S^{n-1}}\left(\frac{|u\cdot w|\rho_K(w)}{|v\cdot u|\rho_K(v)}\right)^{-p}\rho_K(w)^{n}dw\right)^{-1}\rho^{n}_{\mathrm{I}_pK}(u)du\nonumber\\
&=&\frac{1}{nV(K)}\int_{S^{n-1}}\left(\frac{1}{V(K)}\int_{K}\left(\frac{|u\cdot x|\rho_K(x)}{|v\cdot u|\rho_K(v)}\right)^{-p}dx\right)^{-\frac{1}{p}\times p}\rho^{n}_{\mathrm{I}_pK}(u)du.
\end{eqnarray}

For $K \in \mathcal{K}_o^n$ with $\mathcal{C}^1$ boundary and fixed $v\in S^{n-1}$, $\mathcal{H}^{n-1}$-a.e. $u\in S^{n-1}$,
\begin{equation*}
\lim_{p\rightarrow-\infty}\left(\frac{1}{V(K)}\int_{K}\left(\frac{|u\cdot x|\rho_K(x)}{|v\cdot u|\rho_K(v)}\right)^{-p}dx\right)^{-\frac{1}{p}}=\operatorname{ess\;sup}\left\{\frac{|u\cdot x|\rho_K(x)}{|v\cdot u|\rho_K(v)}:\;x\in K\right\}=\frac{h_K(u)}{|v\cdot u|\rho_K(v)}>1.
\end{equation*}
Therefore, for $\mathcal{H}^{n-1}$-a.e. $u\in S^{n-1}$,
\begin{eqnarray}\label{3.3}
&&\lim_{p\rightarrow-\infty}\left(\frac{1}{V(K)}\int_{K}\left(\frac{|u\cdot x|\rho_K(x)}{|v\cdot u|\rho_K(v)}\right)^{-p}dx\right)^{-\frac{1}{p}\times p}=0.
\end{eqnarray}
Moreover, by \cite[P.2]{LZ97}, for small enough $r>0,$
\begin{eqnarray}\label{3.4}
\limsup_{p\rightarrow-\infty} \rho_{\mathrm{I}_pK} \leq \lim_{p\rightarrow-\infty}\rho_{\mathrm{I}_p(B_r)}=\lim_{p\rightarrow-\infty}\rho_{\Gamma^{\ast}_{-p}(B_r)}=\rho_{(B_r)^{\ast}}=\frac{1}{r}.
\end{eqnarray}
By (\ref{3.2}), (\ref{3.3}) and (\ref{3.4}), we have
\begin{eqnarray}\label{3.5}
\lim_{p\rightarrow-\infty}\left(\frac{\rho_{\mathrm{I}_p^2K(v)}}{\rho_K(v)}\right)^p=0.
\end{eqnarray}
By (\ref{2.3}), (\ref{3.5}) and that $K \in \mathcal{K}_o^n$ has $\mathcal{C}^1$ boundary, $\mathcal{I}_p(K,\cdot)$ converges to zero mesure as $p\rightarrow -\infty$.
\qed
\end{proof}

\begin{proposition}Let $K \in \mathcal{K}_o^n$ and $V(K)=2$. Then
$\frac{2|p|}{n^2V(\mathrm{I}_pK)}\mathcal{I}_p(K,\cdot)$ converges to cone-volume measure $V_K(\cdot)$ as $p\rightarrow 0$.
\end{proposition}
\begin{proof}

Let $\mathrm{I}_0K$ be the star body defined by its radial function
\begin{eqnarray}\label{Drho}
\rho(\mathrm{I}_0K,u):=\lim_{p\rightarrow 0}\rho(\mathrm{I}_pK,u)=\lim_{p\rightarrow 0}\left(\frac{1-p}{2}\int_K|x\cdot u|^{-p}dx\right)^{\frac{1}{p}},\;\;\forall\;u\in S^{n-1}.
\end{eqnarray}
By $V(K)=2$ and simple computations,
\begin{eqnarray}
\rho(\mathrm{I}_0K,u)=\exp\left\{-\frac{1}{2}\int_K\ln|u\cdot x|dx-1\right\},\;\;\forall\;u\in S^{n-1}.
\end{eqnarray}
By (\ref{Drho}), we have
\begin{eqnarray}\label{3.11}
V(\mathrm{I}_0K)=\lim_{p\rightarrow 0}V(\mathrm{I}_pK).
\end{eqnarray}

By $V(K)=2$, (\ref{3.2}) and (\ref{Drho}), we have
\begin{eqnarray}\label{3.12}
\lim_{p\rightarrow0}\left(\frac{\rho_{\mathrm{I}_p^2K(v)}}{\rho_K(v)}\right)^p
&=&\lim_{p\rightarrow0}\int_{S^{n-1}}\left(n\int_{K}\left(\frac{|u\cdot x|\rho_K(x)}{|v\cdot u|\rho_K(v)}\right)^{-p}dx\right)^{-1}\rho^{n}_{\mathrm{I}_pK}(u)du\nonumber\\
&=&\frac{1}{2n}\int_{S^{n-1}}\rho^n_{\mathrm{I}_0K}(u)du=\frac{1}{2}V(\mathrm{I}_0K).
\end{eqnarray}
Therefore, by (\ref{2.3}), (\ref{3.11}) and (\ref{3.12}), for any spherical Borel set $\eta\subset S^{n-1}$,
\begin{eqnarray}
\lim_{p\rightarrow 0}\frac{2|p|}{n^2V(\mathrm{I}_pK)}\mathcal{I}_p(K,\eta)
&=&\lim_{p\rightarrow0}\frac{2(n-p)}{n^2V(\mathrm{I}_pK)}\int_{{\boldsymbol\alpha}_K^{\ast}(\eta)}\rho^{n-p}_K(v)\rho^{p}_{\mathrm{I}_p^2K}(v)dv\nonumber\\
&=&\lim_{p\rightarrow0}\frac{2(n-p)}{n^2V(\mathrm{I}_pK)}\int_{{\boldsymbol\alpha}_K^{\ast}(\eta)}\rho^{n}_K(v)\left(\frac{\rho_{\mathrm{I}_p^2K}(v)}{\rho_K(v)}\right)^pdv\nonumber\\
&=&\frac{1}{n}\int_{{\boldsymbol\alpha}_K^{\ast}(\eta)}\rho^{n}_K(v)dv=V_K(\eta).\nonumber
\end{eqnarray}
\qed
\end{proof}

\section{The Minkowski problem via maximization}
If $\mu$ is a nonzero finite even Borel measure on $S^{n-1}$, then the {\it entropy functional} of $\mu$, $E_\mu: C_e^{+}\left(S^{n-1}\right) \rightarrow \mathbb{R}$, is defined by
\begin{equation}\label{Emuf}
E_\mu(f)=-\frac{1}{|\mu|} \int_{S^{n-1}} \log f(v) d \mu(v).
\end{equation}
For $K\in\mathcal{K}_e^n$, let $E_{\mu}(K)$ denote $E_{\mu}(h_K)$.

Let $p\in(-\infty,0)\cup (0,1)$ and $\mu$ be a  nonzero finite even Borel measure on $S^{n-1}$. Define the functional
$$
\Phi_{\mu,p}: C_e^{+}\left(S^{n-1}\right) \rightarrow \mathbb{R}
$$
by
\begin{equation}\label{phimupf}
\Phi_{\mu,p}(f)=\frac{p}{n(n-p)} \log V(\mathrm{I}_p[f])+E_\mu(f).
\end{equation}
By (\ref{IpK}) and the definition of $E_\mu(f)$,  $\Phi_{\mu,p}$ is homogenous of degree
0, i.e., $\Phi_{\mu,p}(c f)=\Phi_{\mu,p}(f)$, for all real $c>0$.

\medskip

\noindent{\bf Maximization Problem I.} Given a  nonzero finite even Borel measure $\mu$ on $S^{n-1}$, does there exist $f_0 \in C_e^{+}\left(S^{n-1}\right)$ such that
$$
\Phi_{\mu,p}\left(f_0\right)=\sup \left\{\Phi_{\mu,p}(f): f \in C_e^{+}\left(S^{n-1}\right)\right\} ?
$$

Note that the set of support functions of convex bodies in $\mathcal{K}_e^n$ is a convex sub-cone of $C_e^{+}\left(S^{n-1}\right)$, thus we can restrict the functional $
\Phi_{\mu,p}: \mathcal{K}_e^n \rightarrow \mathbb{R}$,
which maps $K$ to $\Phi_{\mu,p}\left(h_K\right)$, and we will denote it as $\Phi_{\mu,p}(K)$ in brief. This leads to the following maximization problem.

\medskip

\noindent {\bf Maximization Problem II.} Given a  nonzero finite even Borel measure $\mu$ on $S^{n-1}$, does there exist $K_0 \in \mathcal{K}_e^n$ such that
$$
\Phi_{\mu,p}\left(K_0\right)=\sup \left\{\Phi_{\mu,p}(K): K \in \mathcal{K}_e^n\right\} ?
$$

We first show that the solution of Maximization Problem II is the solution of Maximization Problem I.

\begin{lemma} Given a nonzero finite even Borel measure $\mu$, suppose there is a convex body $K_0 \in \mathcal{K}_e^n$ such that
\begin{eqnarray}\label{PhimupK}
\Phi_{\mu,p}\left(K_0\right)=\sup \left\{\Phi_{\mu,p}(K): K \in \mathcal{K}_e^n\right\} .
\end{eqnarray}
Then
\begin{eqnarray}\label{phimu}
\Phi_{\mu,p}\left(K_0\right)=\sup \left\{\Phi_{\mu,p}(f): f \in C_e^{+}\left(S^{n-1}\right)\right\} .
\end{eqnarray}
\end{lemma}
\begin{proof}
Let $f\in C_e^+(S^{n-1})$, then $h_{[f]}(u)\leq f(u)$ for any $u\in S^{n-1}$. Moreover, $\left[h_{[f]}\right]=[f]$. Thus,
\begin{eqnarray}
\Phi_{\mu,p}\left([f]\right)&=&\Phi_{\mu,p}\left(h_{[f]}\right)\nonumber\\
&=&\frac{p}{n(n-p)} \log V(\mathrm{I}_p[f])-\frac{1}{|\mu|} \int_{S^{n-1}} \log h_{[f]}(u) d \mu(u)\nonumber\\
&\geq& \frac{p}{n(n-p)} \log V(\mathrm{I}_p[f])-\frac{1}{|\mu|} \int_{S^{n-1}} \log f(u) d \mu(u)\nonumber\\
&=& \Phi_{\mu,p}\left(f\right).
\end{eqnarray}
Thus, Maximization Problem I can be obtained in the support function of some $K_0\in \mathcal{K}_e^n$. The desired equality (\ref{phimu}) now follows from (\ref{PhimupK}).
\qed
\end{proof}

The following lemma shows that the solution of Maximization Problem I is the solution of the affine Minkowski problem for the $p$-affine dual curvature measure.

\begin{lemma} Let $p\in (-\infty,0)\cup(0,1)$. Given a nonzero finite even Borel measure $\mu$, suppose there is a convex body $K_0 \in \mathcal{K}_e^n$ such that
\begin{eqnarray}\label{leftK0}
\Phi_{\mu,p}\left(K_0\right)=\sup \left\{\Phi_{\mu,p}(f): f \in C_e^{+}\left(S^{n-1}\right)\right\}.
\end{eqnarray}
Then $\mu=\mathcal{I}_p\left(c K_0, \cdot\right)$ for a positive real number $c$.
\end{lemma}

\begin{proof}
 Suppose $K_0$ is the maximizer in (\ref{leftK0}). Since the functional $\Phi_{\mu,p}(K)$ is homogenous of degree $0$ and $V\left(\mathrm{I}_pK\right)$  is homogenous of degree $n(n-p)/p$ by (\ref{IpK}), we can find a positive real number $c$ such that
\begin{eqnarray}\label{VIPck}
V\left(\mathrm{I}_p(c K_0)\right)=\frac{|p||\mu|}{n(n-p)}
\end{eqnarray}
and $cK_0$ is also the maximizer in (\ref{leftK0}). Denote $K_t=\left[cK_0,g\right]_t$, where $g\in C_e(S^{n-1})$.
By Theorem \ref{dtrhoKt} and (\ref{VIPck}),
$$
\begin{aligned}
0 & =\left.\frac{d}{d t} \Phi_{\mu,p}\left(K_t\right)\right|_{t=0} \\
& =\left.\frac{d}{d t}\left(\frac{p}{n(n-p)} \log V\left(\mathrm{I}_p K_t\right)-\frac{1}{|\mu|} \int_{S^{n-1}} \log h_{K_t} d \mu\right)\right|_{t=0} \\
& =\frac{p}{n(n-p)} \cdot \frac{1}{V(\mathrm{I}_p(cK_0))}\left(\left.\frac{d}{d t} V\left(\mathrm{I}_p K_t\right)\right|_{t=0}\right)-\frac{1}{|\mu|} \int_{S^{n-1}} g(u) d \mu \\
& =\frac{|p|}{n(n-p)} \cdot \frac{1}{V(\mathrm{I}_p(cK_0))} \int_{S^{n-1}} g(u) d \mathcal{I}_p\left(c K_0, u\right)-\frac{1}{|\mu|} \int_{S^{n-1}} g(u) d \mu\\
& =\frac{1}{|\mu|} \int_{S^{n-1}} g(u) d \mathcal{I}_p\left(c K_0, u\right)-\frac{1}{|\mu|} \int_{S^{n-1}} g(v) d \mu,
\end{aligned}
$$
which implies
$$\int_{S^{n-1}} g(u) d\mathcal{I}_p\left(c K_0, u\right)=\int_{S^{n-1}} g(v) d \mu$$
for every $g \in C_e\left(S^{n-1}\right)$. Thus
$\mu=\mathcal{I}_p\left(c K_0, \cdot\right)$.
\qed
\end{proof}

\section{Solutions to the affine Minkowski problem}

\begin{lemma}\label{T2}
Suppose $p\in (-\infty,0)\cup(0,1)$. For all $K \in \mathcal{K}_e^n$, there exists constant $c_{n,p}>0$, independent of the body $K$, such that
\begin{eqnarray}\label{IpKandIK}
c_{n,p}V(K)^{\frac{1-p}{p}} \rho(\mathrm{I} K, u) \leq \rho\left(\mathrm{I}_p K, u\right) \leq 2^{\frac{p-1}{p}}V(K)^{\frac{1-p}{p}} \rho(\mathrm{I} K, u),\;\;\;\forall u \in S^{n-1}.
\end{eqnarray}
\end{lemma}

\begin{proof}
The following two facts can be found in the paper \cite{MP89}. For a measurable function $f: \mathbb{R}^n \rightarrow [0,1]$ and $Q \in \mathcal{K}_e^n$, the function
$$
F_1(q):=\left(\frac{\int_{\mathbb{R}^n} \rho_Q(x)^{-q} f(x) d x}{\int_Q \rho_Q(x)^{-q} d x}\right)^{1 /(n+q)}
$$
is increasing on $(-n, \infty)$.
Suppose $\psi: [0,+\infty) \rightarrow [0,+\infty)$ satisfies $\psi(0)=0$, $\psi$ and $\psi(x) / x$ are increasing on $(0, t]$, and $\psi(x)=\psi(t)$ for $x \geqslant t$. Let $g: [0,+\infty) \rightarrow [0,+\infty)$ be a decreasing, continuous function which vanishes at $\psi(t)$. Then
$$
F_2(q):=\left(\frac{\int_0^{\infty} g(\psi(x)) x^q d x}{\int_0^{\infty} h(x) x^q d x}\right)^{1 /(1+q)}
$$
is a decreasing function on $(-1, \infty)$.

To prove the right inequality in (\ref{IpKandIK}), let $Q:=[-1,1] \subset \mathbb{R}$ and
$f(x):=\frac{A_{K, u}(x)}{A_{K, u}(0)}$,
 where
$$A_{K,u}(x)=V_{n-1}\left(K\cap \{z\in\mathbb{R}^n:\;z\cdot u =x\}\right).$$
Brunn's theorem shows that ${A_{K, u}(\cdot)}^{1/ (n-1)}$ is concave on $[-h(K, u), h(K, u)]$, which implies that the even function ${A_{K, u}(\cdot)}$ is decreasing on $[0,\infty)$,
and thus $f(x)\leq f(0)=1,~\forall x\in\mathbb{R}.$ Therefore, $F_1(-p) \leqslant$ $F_1(0)$ when $p\in (0,1)$, that is
$$
\left(\frac{(1-p) \int_{\mathbb{R}}|x|^{-p} A_{K, u}(x) d x}{2V_{n-1}\left(K \cap u^{\perp}\right)}\right)^{1 /(1-p)} \leqslant \frac{V(K)}{2V_{n-1}\left(K \cap u^{\perp}\right)}.
$$
And $F_1(-p) \geqslant$ $F_1(0)$ when $p\in (-\infty,0)$, that is
$$
\left(\frac{(1-p) \int_{\mathbb{R}}|x|^{-p} A_{K, u}(x) d x}{2V_{n-1}\left(K \cap u^{\perp}\right)}\right)^{1 /(1-p)} \geqslant \frac{V(K)}{2V_{n-1}\left(K \cap u^{\perp}\right)}.
$$
In the above two cases, by Fubini's theorem and (\ref{DIP}), we can get
$$\rho\left(\mathrm{I}_p K, u\right) \leqslant 2^{\frac{p-1}{p}}V(K)^{\frac{1-p}{p}} \rho(\mathrm{I} K, u).$$

To establish the left inequality in (\ref{IpKandIK}), let $g(x)=(1-x)^{n-1} \mathbb{I}_{[0,1]}(x), x \geqslant 0$ and
$$\psi(x)=1-\left(\frac{A_{K, u}(x)}{A_{K, u}(0)}\right)^{\frac{1}{n-1}}$$ for some given $u \in S^{n-1}$, ($\mathbb{I}$ denotes the characteristic function). Brunn's theorem shows that $\psi$ is a convex function on $[0, h(K, u)]$. Therefore $g$ and $\psi$ satisfy the above conditions to guarantee the monotonicity of $F_2$. Hence $F_2(-p) \geqslant F_2(0)$ when $p\in (0,1)$, i.e.,
\begin{eqnarray}\label{AKu}
\left(\frac{\int_0^{\infty} A_{K, u}(x) x^{-p} d x}{V_{n-1}\left(K \cap u^{\perp}\right) B(1-p, n)}\right)^{1 /(1-p)} \geqslant \frac{nV(K)}{2V_{n-1}\left(K \cap u^{\perp}\right)},
\end{eqnarray}
where $B(\cdot,\cdot)$ denotes the beta function. And $F_2(-p) \leqslant F_2(0)$ when $p\in (-\infty,0)$,  i.e.,
\begin{eqnarray}\label{AKu1}
\left(\frac{\int_0^{\infty} A_{K, u}(x) x^{-p} d x}{V_{n-1}\left(K \cap u^{\perp}\right)B(1-p, n)}\right)^{1 /(1-p)} \leqslant \frac{nV(K)}{2V_{n-1}\left(K \cap u^{\perp}\right)}.
\end{eqnarray}

By (\ref{DIP}), (\ref{AKu}) and (\ref{AKu1}), we obtain, for $p\in (-\infty,0)\cup (0,1),$
$$
\rho\left(\mathrm{I}_p K, u\right) \geqslant 2^{\frac{p-1}{p}}\left((1-p)\beta(1-p,n)n^{1-p}\right)^{1 / p}V(K)^{\frac{1-p}{p}} \rho(\mathrm{I} K, u).
$$
Let $$C_{n,p}:=2^{\frac{p-1}{p}}\left((1-p)\beta(1-p,n)n^{1-p}\right)^{1 / p}.$$ We get the desired results.
\qed
\end{proof}

The proof of the following Lemma \ref{L7.3} follows along the same lines as that of \cite[Lemma 4.2]{BLYZZ19}, and thus we omit the proof. The difference between the two lemmas is that $q\in (1,n)$ in \cite[Lemma 4.2]{BLYZZ19} and $q=n$ in the following Lemma \ref{L7.3}.

\begin{lemma}\label{L7.3}  Let $\epsilon_0>0$. Suppose that $\left(e_{1l}, \ldots, e_{nl}\right)$, where $l=1,2, \ldots$, is a sequence of ordered orthonormal basis of $\mathbb{R}^n$ converging to the ordered orthonormal basis $\left(e_1, \ldots, e_n\right)$. Suppose also that $\left(a_{1l}, \ldots, a_{nl}\right)$ is a sequence of $n$-tuples satisfying $0<a_{1l} \leq a_{2l} \leq \cdots \leq a_{nl}$ for all $l$, and there exists an $\epsilon_0>0$, such that $a_{nl}>\epsilon_0$ for all $l$. For each $l=1,2, \ldots$, let
$$
Q_l=\left\{x \in \mathbb{R}^n:\left|x \cdot e_{1l}\right|^2/a_{1l}^2+\cdots+\left|x \cdot e_{nl}\right|^2/a_{nl}^2\leq 1\right\}
$$
denote the ellipsoid generated by the $\left(e_{1l}, \ldots, e_{nl}\right)$ and $\left(a_{1l}, \ldots, a_{nl}\right)$. If $\mu$ is an even Borel measure on $S^{n-1}$ that satisfies the strict subspace concentration inequality (\ref{1.4}), then there exists $t_0$, $\delta_0$, $l_0>0$ and $c_{n, \epsilon_0, \delta_0, t_0}$, independent of $l$, such that for each $l>l_0$,
\begin{eqnarray}\label{Emu}
E_\mu\left(Q_l\right) \leq-\frac{\log \left(a_{1l} \cdots a_{n-1,l}\right)}{n}+t_0 \log a_{1l}+c_{n, \epsilon_0,\delta_0, t_0}.
\end{eqnarray}

\end{lemma}

\begin{theorem}
 Suppose $p\in (-\infty,0)\cup(0,1)$ and $\mu$ is a nonzero finite even Borel measure on $S^{n-1}$. If $\mu$ satisfies the strict subspace concentration inequality  (\ref{1.4}), then there is $K_0 \in \mathcal{K}_e^n$ such that
\begin{eqnarray}\label{Phimup}
\Phi_{\mu,p}\left(K_0\right)=\sup \left\{\Phi_{\mu,p}(K): K \in \mathcal{K}_e^n\right\} .
\end{eqnarray}
\end{theorem}

\begin{proof}
Suppose there is a sequence $\left\{K_j\right\} \subset \mathcal{K}_e^n$ such that
\begin{eqnarray}\label{Maxsequenc}
\lim _{j \rightarrow \infty} \Phi_{\mu,p}\left(K_j\right)=\sup \left\{\Phi_{\mu,p}(K): K \in \mathcal{K}_e^n\right\} .
\end{eqnarray}

 Since $\Phi_{\mu,p}(c K)=\Phi_{\mu,p}(K)$ for any $c>0$, we can assume $\max\{|x|:\;x\in K_j\}=1$, for all $j$.
By Blaschke selection theorem, there is a subsequence of $\left\{K_j\right\}$, still denoted as $\left\{K_j\right\}$ for convenience, satisfying
$
\lim _{j \rightarrow \infty} K_j=K_0 .
$
It is easy to see $K_0$ is an origin-symmetric compact  convex set.

Next, we prove that $K$ contains the origin in its interior. Otherwise, $K_0$ is contained in a proper subspace of $\mathbb{R}^n$. It follows from John's theorem that  there is an ellipsoid $Q_j$ such that
\begin{eqnarray}\label{QjKj}
Q_j \subset K_j \subset \sqrt{n} Q_j, \quad j=1,2, \ldots.
\end{eqnarray}
Denote $Q_j$ by
$$
Q_j=\left\{x \in \mathbb{R}^n: \left|x \cdot e_{1 j}\right|^2/a_{1 j}^2+\left|x \cdot e_{2 j}\right|^2/a_{2 j}^2+\cdots+\left|x \cdot e_{n j}\right|^2/a_{n j}^2\leq 1\right\},
$$
where $\left\{e_{i j}\right\}_{i=1}^n$ is an orthonormal basis of $\mathbb{R}^n$ such that $
0<a_{1 j} \leq a_{2 j} \leq \cdots \leq a_{n j}$.

According to the property of intersection bodies in \cite[Corollary 8.1.7]{Gardner},
$$
\mathrm{I} Q_j=\frac{\omega_{n-1}V(Q_j)}{\omega_n} Q_j^{\ast}.
$$
 Thus we have
\begin{eqnarray}\label{VleftI}
V\left(\mathrm{I} Q_j\right)=\left(\frac{\omega_{n-1}V(Q_j)}{\omega_n}\right)^nV(Q_j^{\ast})=\frac{\omega^n_{n-1}}{\omega^{n-2}_n}V(Q_j)^{n-1}.
\end{eqnarray}
 Therefore, when $p\in (0,1)$, by (\ref{IpK}), (\ref{QjKj}), (\ref{T2}) and (\ref{VleftI}), there exist $\hat{c}_{n,p}>0$ and $\check{c}_{n,p}>0$ such that
\begin{eqnarray}\label{VIp}
\frac{p}{n(n-p)} \log V\left(\mathrm{I}_p K_j\right)
&\leq&\frac{p}{n(n-p)} \log V\left(\mathrm{I}_p \left(\sqrt{n}Q_j\right)\right)\\
&\leq&\frac{p}{n(n-p)} \log\left(2^{\frac{n(p-1)}{p}}\sqrt{n}^{\frac{n(n-p)}{p}} V(Q_j)^{\frac{n(1-p)}{p}}V\left(\mathrm{I}Q_j\right)\right)\nonumber\\
&=& \hat{c}_{n,p}+\frac{p}{n(n-p)} \log V(Q_j)^{\frac{n-p}{p}}\nonumber\\
&=&\check{c}_{n,p}+\frac{1}{n} \log \left(a_{1 j} a_{2 j} \cdots a_{n j}\right).\nonumber
\end{eqnarray}
And when $p\in (-\infty,0)$,
by (\ref{IpK}), (\ref{QjKj}), (\ref{T2}) and (\ref{VleftI}), there exist  $c_{n,p}>0$, $\bar{c}_{n,p}>0$ and $\tilde{c}_{n,p}$ such that
\begin{eqnarray}\label{VIp1}
\frac{p}{n(n-p)} \log V\left(\mathrm{I}_p K_j\right)
&\leq&\frac{p}{n(n-p)} \log V\left(\mathrm{I}_p (\sqrt{n}Q_j)\right)\\
&\leq&\frac{p}{n(n-p)} \log\left( c_{n,p}^n\sqrt{n}^{\frac{n(n-p)}{p}}V(Q_j)^{\frac{n(1-p)}{p}}V\left(\mathrm{I}Q_j\right)\right)\nonumber\\
&=& \bar{c}_{n,p}+\frac{p}{n(n-p)} \log V(Q_j)^{\frac{n-p}{p}}\nonumber\\
&=&\tilde{c}_{n,p}+\frac{1}{n} \log \left(a_{1 j} a_{2 j} \cdots a_{n j}\right).\nonumber
\end{eqnarray}

We turn to estimate the upper bound of $E_{\mu}$. By the definition (\ref{Emuf}),
we have $E_\mu(K_j) \leq E_\mu\left(Q_j\right)$.
By compactness and orthogonality of $\{e_{ij}\}_{i
=1}^n$ for each $j$,  there is a subsequence of $\left(e_{1 j}, \ldots, e_{n j}\right)$ converging to $\left(e_1, \ldots, e_n\right)$, where $\left\{e_i\right\}_{i=1}^n$ is an orthonormal basis of $\mathbb{R}^n$. Similarly, $\left\{a_{i j}\right\}$ has convergent subsequence, for $i=1, \ldots, n$.
Therefore, without loss of generality,  we suppose that $\left(e_{i j}, \ldots, e_{n j}\right)$ converges to $\left(e_1, \ldots, e_n\right)$ and $a_{i j}$ converges to $a_i$, for every $i$.

Since $Q_j\subset K_j \subset \sqrt{n} Q_j$ and $\max\{|x|:\;x\in K_j\}=1$, we have
\begin{eqnarray}\label{sqrt}
\frac{1}{\sqrt{n}}\leq a_{n j} \leq 1.
 \end{eqnarray}
 According to (\ref{Emu}), there exist positive real numbers $t_0$, $\delta_0$, $j_0$ and constant $c_{n,\delta_0, t_0}$, independent of $j$, such that, for $j>j_0$,
\begin{eqnarray}\label{EmuQj}
E_\mu\left(Q_j\right)\leq-\frac{1}{n} \log \left(a_{1 j} a_{2 j} \cdots a_{n-1,j}\right)+t_0 \log a_{1 j}+c_{n,\delta_0,t_0}.
\end{eqnarray}

As $K_0$ is contained in a proper subspace and $a_{i j}$ converges to $a_i$, there is $1<k\leq n$ satisfying
$$
a_1=a_2=\cdots=a_{k-1}=0, \quad a_k, \ldots, a_n>0 .
$$
By (\ref{phimupf}), (\ref{VIp}), (\ref{VIp1}), (\ref{EmuQj}) and (\ref{sqrt}), for $j>j_0$
\begin{align*}
\Phi_{\mu,p}\left(K_j\right)&=\frac{p}{n(n-p)} \log V\left(\mathrm{I}_p K_j\right)+ E_\mu\left(K_j\right)\\
&\leq\frac{p}{n(n-p)} \log V\left(\mathrm{I}_p K_j\right)+ E_\mu\left(Q_j\right)\\
&\leq \frac{1}{n} \log a_{n j}+ t_0 \log a_{1 j}+\check{c}_{n,p}+\tilde{c}_{n,p}+c_{n,\delta_0, t_0}\rightarrow-\infty\;\;{\rm as}\;j \rightarrow \infty.
\end{align*}
But
$$
-\infty=\lim _{j \rightarrow \infty} \Phi_{\mu,p}\left(K_j\right) \geq \Phi_{\mu,p}\left(B^n\right)=\frac{p}{n(n-p)} \log V(\mathrm{I}_pB^n)
$$
which yields the contradiction.
Thus $K_0$ has nonempty interior, that is $K_0 \in \mathcal{K}_e^n$. \qed
\end{proof}

\section{Necessity of a subspace concentration
bound on $p$-affine dual curvature measures}
Our main result of the section treats the necessity of a subspace concentration bound on  $p$-affine dual curvature measures for $p\in (0,1)$.
The superlevel sets of a function $f: \mathbb{R}^n \rightarrow \mathbb{R}$ are given by $$L^{+}(f,h)=\left\{x \in \mathbb{R}^n: f(x) \geq h\right\},\;\;h \in \mathbb{R}.$$
We say that $f$ is {\it unimodal} if every superlevel set of $f$ is closed and convex.

\begin{lemma}\label{L6.2}\cite[Lemma 3.1]{BHP18}
 Let $f: \mathbb{R}^n \rightarrow [0,+\infty]$ be a unimodal function, such that $f(x)=f(-x)$ for every $x \in \mathbb{R}^n$ and let $f$ be integrable on compact, convex sets. Let $K \subset \mathbb{R}^n$ be a compact, convex set with $\operatorname{dim} K=k$. Then for $\lambda \in[0,1]$,
\begin{eqnarray}
\int_{\lambda K+(1-\lambda)(-K)} f(x) \mathrm{d} \mathcal{L}^k(x) \geq \int_K f(x) \mathrm{d} \mathcal{L}^k(x) .
\end{eqnarray}
Moreover, equality holds if and only if for every $h>0$
\begin{eqnarray}
V_k\left([\lambda K+(1-\lambda)(-K)] \cap L^{+}(f,h)\right)=V_k\left(K \cap L^{+}(f,h)\right).
\end{eqnarray}
\end{lemma}

\noindent{\bf Proof of Theorem \ref{T1.2}}  For $y\in K \mid \xi$, denote $\overline{y}=\rho_{K \mid \xi}(y) y$,
$$F_{y}=\operatorname{conv}\left\{0, K \cap\left(\overline{y}+\xi^{\perp}\right)\right\},\;\;{\rm and}\;\;M_y=\operatorname{conv}\left\{K \cap \xi^{\perp}, K \cap\left(\overline{y}+\xi^{\perp}\right)\right\}.$$
Observe that
\begin{eqnarray}
M_{y} \cap \left(\overline{y}+\xi^{\perp}\right)=F_{y} \cap\left(\overline{y}+\xi^{\perp}\right)=K \cap\left(\overline{y}+\xi^{\perp}\right).
\end{eqnarray}

By (\ref{np2}), Fubini's theorem and the fact that $M_{y} \cap\left(y+\xi^{\perp}\right) \subseteq K \cap\left(y+\xi^{\perp}\right)$ we may write
\begin{eqnarray}\label{6.2}
\mathcal{I}_p\left(K, S^{n-1}\right)
 & =&\frac{(n-p)^2}{p}\int_K\rho^p_{\mathrm{I}^2_pK}(x)dx\nonumber\\
 & =&\frac{(n-p)^2}{p} \int_{K \mid \xi}\left(\int_{K \cap\left(y+\xi^{\perp}\right)}\rho^p_{\mathrm{I}^2_pK}(z) \mathrm{~d} \mathcal{L}^{n-k}(z)\right) \mathrm{d} \mathcal{L}^k(y)\nonumber\\
& \geq& \frac{(n-p)^2}{p}\int_{K \mid \xi}\left(\int_{M_y \cap\left(y+\xi^{\perp}\right)}\rho^p_{\mathrm{I}^2_pK}(z) \mathrm{~d} \mathcal{L}^{n-k}(z)\right) \mathrm{d} \mathcal{L}^k(y) .
\end{eqnarray}

In order to estimate the inner integral, let $y \in K \mid \xi$, $y \neq 0$, and for abbreviation we set $\lambda=\rho_{K \mid \xi}(y)^{-1} \leq 1$. Then by the symmetry of $K$ we find

\begin{eqnarray}\label{8.3}
M_{y} \cap\left(y+\xi^{\perp}\right) & \supseteq& \lambda\left(K \cap\left(\overline{y}+\xi^{\perp}\right)\right)+(1-\lambda)\left(K \cap \xi^{\perp}\right)\nonumber\\
&\supseteq&\lambda\left(K \cap\left(\overline{y}+\xi^{\perp}\right)\right)+(1-\lambda)\left(\frac{1}{2}\left(K \cap\left(\overline{y}+\xi^{\perp}\right)\right)+\frac{1}{2}\left(-\left(K \cap\left(\overline{y}+\xi^{\perp}\right)\right)\right)\right)\nonumber\\
&=&\frac{1+\lambda}{2}\left(K \cap\left(\overline{y}+\xi^{\perp}\right)\right)+\frac{1-\lambda}{2}\left(-\left(K \cap\left(\overline{y}+\xi^{\perp}\right)\right)\right).
\end{eqnarray}

 Theorem in \cite[P921]{Berck09} shows that for $0\neq p<1$, the $L_p$-intersection body of an origin-symmetric convex body is an origin-symmetric convex body. Thus, $\mathrm{I}^2_pK$ is an origin-symmetric convex body for $K\in\mathcal{K}_e^n$. Thus the superlevel sets of the function $\rho_{\mathrm{I}^2_pK}^{p}(z)$ are  closed and convex for $p\in(0,1)$, if we define $\rho_{\mathrm{I}^2_pK}^{p}(0)=+\infty$. Thus, the function $\rho_{\mathrm{I}^2_pK}^{p}(\cdot)$ is unimodal.  By (\ref{8.3}) and Lemma \ref{L6.2}, we get
\begin{eqnarray}\label{4.2}
\int_{M_{y} \cap\left(y+\xi^{\perp}\right)}\rho_{\mathrm{I}^2_pK}^{p}(z)  \mathrm{~d} \mathcal{L}^{n-k}(z) \geq \int_{K \cap\left(\overline{y}+\xi^{\perp}\right)}\rho_{\mathrm{I}^2_pK}^{p}(z)  \mathrm{~d} \mathcal{L}^{n-k}(z),
\end{eqnarray}
for every $y \in K \mid \xi$, $y \neq 0$. Together with (\ref{6.2}) we obtain the lower bound
\begin{eqnarray}\label{6.3}
\mathcal{I}_p\left(K, S^{n-1}\right) \geq  \frac{(n-p)^2}{p} \int_{K \mid \xi}\left(\int_{K \cap\left(\overline{y}+\xi^{\perp}\right)}\rho_{\mathrm{I}^2_pK}^{p}(z) \mathrm{~d} \mathcal{L}^{n-k}(z)\right) \mathrm{d} \mathcal{L}^k(y) .
\end{eqnarray}

Next, we turn to evaluate $\mathcal{I}_p\left(K, S^{n-1} \cap \xi\right)$. Note that for $x \in K$, we have $x /|x| \in {\boldsymbol\alpha}_K^{\ast}\left(S^{n-1} \cap \xi\right)$ if and only if the boundary point $\rho_K(x) x$ has an outer unit normal in $\xi$. Hence,
$$
\begin{aligned}
\{x \in K & \left.: x /|x| \in {\boldsymbol\alpha}_K^{\ast}\left(S^{n-1} \cap \xi\right)\right\} \cup\{0\}=\bigcup_{\bar{y} \in \partial(K \mid \xi)} \operatorname{conv}\left\{0,\left(\bar{y}+\xi^{\perp}\right) \cap K\right\}
\end{aligned}
$$
Using (\ref{np2}) and Fubini's theorem, we obtain
\begin{eqnarray}\label{4.4}
&&\mathcal{I}_p\left(K, S^{n-1} \cap \xi\right) \\
&=&\frac{(n-p)^2}{p}\int_{K \mid \xi}\left(\int_{F_{y} \cap\left(y+\xi^{\perp}\right)}\rho_{\mathrm{I}^2_pK}^{p}(z)\mathrm{~d} \mathcal{L}^{n-k}(z)\right) \mathrm{d} \mathcal{L}^k(y)\nonumber \\
&=&\frac{(n-p)^2}{p}\int_{K \mid \xi}\left(\int_{\rho_{K \mid \xi}(y)^{-1}\left(K \cap\left(\overline{y}+\xi^{\perp}\right)\right)}\rho_{\mathrm{I}^2_pK}^{p}(z) \mathrm{~d} \mathcal{L}^{n-k}(z)\right) \mathrm{d} \mathcal{L}^k(y)\nonumber\\
&=&\frac{(n-p)^2}{p}\int_{K \mid \xi} \rho_{K \mid \xi}(y)^{k+p-n}\left(\int_{K \cap\left(\overline{y}+\xi^{\perp}\right)}\rho_{\mathrm{I}^2_pK}^{p}(z)  \mathrm{~d} \mathcal{L}^{n-k}(z)\right) \mathrm{d} \mathcal{L}^k(y).\nonumber
\end{eqnarray}
Since the inner integral is independent of the length of $y \in K \mid L$, and thus it can be viewed as a measurable function $g: S^{n-1} \cap \xi \rightarrow [0,\infty)$. Therefore, using spherical coordinates we obtain
\begin{eqnarray}\label{4.5}
\mathcal{I}_p\left(K, S^{n-1} \cap\xi\right) &=&\frac{(n-p)^2}{p}\int_{K \mid \xi} \rho_{K \mid \xi}(y)^{k+p-n} g(y /|y|) \mathrm{d} \mathcal{L}^k(y)\\
&=&\frac{(n-p)^2}{p}\int_{S^{n-1} \cap \xi} g(u)\left(\int_0^{\rho_{K \mid \xi}(u)} \rho_{K \mid \xi}(r u)^{k+p-n} r^{k-1} \mathrm{~d} r\right) \mathrm{d} \mathcal{H}^{k-1}(u)\nonumber\\
&=&\frac{(n-p)^2}{p}\int_{S^{n-1} \cap \xi} g(u) \rho_{K \mid \xi}(u)^{k+p-n}\left(\int_0^{\rho_{K \mid \xi}(u)} r^{n-p-1} \mathrm{~d} r\right) \mathrm{d} \mathcal{H}^{k-1}(u)\nonumber\\
&=&\frac{n-p}{p}\int_{S^{n-1} \cap \xi} g(u) \rho_{K \mid \xi}(u)^k \mathrm{~d} \mathcal{H}^{k-1}(u).\nonumber
\end{eqnarray}

Applying the same transformation to the right hand side of (\ref{6.3}) gives
\begin{eqnarray}\label{4.6}
\mathcal{I}_p\left(K, S^{n-1}\right) & \geq&  \frac{(n-p)^2}{p}\int_{K \mid \xi} g(y /|y|) \mathrm{d} \mathcal{H}^k(y)\\
& =&\frac{(n-p)^2}{p}\int_{S^{n-1} \cap \xi} g(u)\left(\int_0^{\rho_{K \mid \xi}(u)} r^{k-1} \mathrm{~d} r\right) \mathrm{d} \mathcal{H}^{k-1}(u)\nonumber \\
& =&\frac{(n-p)^2}{p}\frac{1}{k} \int_{S^{n-1} \cap \xi} g(u) \rho_{K \mid \xi}(u)^k \mathrm{~d} \mathcal{H}^{k-1}(u)\nonumber\\
&=& \frac{n-p}{k}\mathcal{I}_p\left(K, S^{n-1} \cap\xi\right),\nonumber
\end{eqnarray}
where the last equation follows from
 (\ref{4.5}). Thus,
\begin{eqnarray}\label{IIp}
\frac{\mathcal{I}_p\left(K, S^{n-1} \cap \xi\right) }{\mathcal{I}_p\left(K, S^{n-1}\right)}\leq \frac{k}{n-p}.
\end{eqnarray}

Applying the equality condition of Lemma \ref{L6.2} to equality in (\ref{4.2}), we get that
$$V_{n-k}\left(M_{y} \cap\left(y+\xi^{\perp}\right) \cap L^+(\rho_{\mathrm{I}^2_pK}^{p},h)\right)=V_{n-k}\left(K \cap\left(\overline{y}+\xi^{\perp}\right) \cap L^+(\rho_{\mathrm{I}^2_pK}^{p},h)\right)$$
for every $h>0.$ However, for sufficiently large $h$, the intersection of $K \cap(\overline{y}+\xi^{\perp}) \subset \partial K$ with $L^+(\rho_{\mathrm{I}^2_pK}^{p},h)$ is empty, while $M_{y} \cap\left(y+\xi^{\perp}\right) \cap L^+(\rho_{\mathrm{I}^2_pK}^{p},h)$ is nonempty when $y$ is sufficiently close to the origin. Therefore, the inequality  (\ref{4.2}) is a strict inequality and so is (\ref{IIp}).
\qed
%



\end{document}